\documentclass{article}

\usepackage{amsmath,amsthm,bm}
\usepackage{enumerate}
\usepackage{comment}
\usepackage{amssymb}
\usepackage{pgfplots}
\usepackage{subcaption}
\usepackage{graphicx}
\usepackage{url}
\usepackage{authblk}
\usetikzlibrary{arrows,positioning,calc,patterns}

\newtheorem{thm}{Theorem}[section]

\newtheorem{prop}[thm]{Proposition}
\newtheorem{lem}[thm]{Lemma}

\begin{document}
\title{On the Representation of Integers by Binary Forms Defined by Means of the Relation $(x + yi)^n = R_n(x, y) + J_n(x, y)i$}
\author{Anton Mosunov}
\affil{University of Waterloo}

\date{}

\maketitle

\begin{abstract}
Let $F$ be a binary form with integer coefficients, degree $d \geq 3$ and non-zero discriminant. Let $R_F(Z)$ denote the number of integers of absolute value at most $Z$ which are represented by $F$. In 2019 Stewart and Xiao proved that $R_F(Z) \sim C_FZ^{2/d}$ for some positive number $C_F$. We compute $C_{R_n}$ and $C_{J_n}$ for the binary forms $R_n(x, y)$ and $J_n(x, y)$ defined by means of the relation
$$
(x + yi)^n = R_n(x, y) + J_n(x, y)i,
$$
where the variables $x$ and $y$ are real.
\end{abstract}

\section{Introduction}
Let $F$ be a binary form with integer coefficients, degree $d \geq 3$ and non-zero discriminant. Let $R_F(Z)$ denote the number of integers of absolute value at most $Z$ which are represented by $F$. In 2019 Stewart and Xiao \cite{stewart-xiao} proved that $R_F(Z) \sim C_FZ^{2/d}$ for some positive number $C_F$. Furthermore, they proved that $C_F = W_FA_F$, where $A_F$ is the area of the \emph{fundamental region}
$$
\{(x, y) \in \mathbb R^2\ \colon\ |F(x, y)| \leq 1\}
$$
of $F$ and $W_F$ is an explicit function of the \emph{rational automorphism group} $\operatorname{Aut} F$ \cite[Theorem 1.2]{stewart-xiao}. The rational automorphism group can be defined as follows: for any matrix $A = \begin{pmatrix}a & b\\c & d\end{pmatrix}$, with rational entries, let
$$
F_A(x, y) = F(ax + by, cx + dy).
$$
Then $\operatorname{Aut} F$ is the set of all matrices $A$ that fix $F$:
$$
\operatorname{Aut} F = \left\{A \in M_{2 \times 2}(\mathbb Q)\ \colon\ F_A = F\right\}.
$$
It is known that $\operatorname{Aut} F$ is a finite subgroup of $\operatorname{GL}_2(\mathbb Q)$, and that every finite subgroup of $\operatorname{GL}_2(\mathbb Q)$ is $\operatorname{GL}_2(\mathbb Q)$-conjugate to one of the $10$ groups listed in \mbox{Table \ref{tab:finite-subgroups}}. In the case when $\operatorname{Aut} F$ is a subgroup of $\operatorname{GL}_2(\mathbb Z)$, the formula for $W_F$ becomes especially simple:
\begin{equation} \label{eq:WF}
W_F = \frac{1}{|\operatorname{Aut} F|}.
\end{equation}

\begin{table}[t]
\centering
\begin{tabular}{| l | l | l | l |}
\hline
Group & Generators & Group & Generators\\
\hline\rule{0pt}{4ex}
$\bm C_1$ &
$\begin{pmatrix}1 & 0\\0 & 1\end{pmatrix}$ &
$\bm D_1$ &
$\begin{pmatrix}0 & 1\\1 & 0\end{pmatrix}$\\\rule{0pt}{5ex}
$\bm C_2$ &
$\begin{pmatrix}-1 & 0\\0 & -1\end{pmatrix}$ &
$\bm D_2$ &
$\begin{pmatrix}0 & 1\\1 & 0\end{pmatrix}, \begin{pmatrix}-1 & 0\\0 & -1\end{pmatrix}$\\\rule{0pt}{5ex}
$\bm C_3$ &
$\begin{pmatrix}0 & 1\\-1 & -1\end{pmatrix}$ &
$\bm D_3$ &
$\begin{pmatrix}0 & 1\\1 & 0\end{pmatrix}, \begin{pmatrix}0 & 1\\-1 & -1\end{pmatrix}$\\\rule{0pt}{5ex}
$\bm C_4$ &
$\begin{pmatrix}0 & 1\\-1 & 0\end{pmatrix}$ &
$\bm D_4$ &
$\begin{pmatrix}0 & 1\\1 & 0\end{pmatrix}, \begin{pmatrix}0 & 1\\-1 & 0\end{pmatrix}$\\\rule{0pt}{5ex}
$\bm C_6$ &
$\begin{pmatrix}0 & -1\\1 & 1\end{pmatrix}$ &
$\bm D_6$ &
$\begin{pmatrix}0 & 1\\1 & 0\end{pmatrix}, \begin{pmatrix}0 & 1\\-1 & 1\end{pmatrix}$\\
\hline
\end{tabular}
\caption{Representatives of equivalence classes of finite subgroups of $\operatorname{GL}_2(\mathbb Q)$ under $\operatorname{GL}_2(\mathbb Q)$-conjugation.}
\label{tab:finite-subgroups}
\end{table}

For a number of families of binary forms with integer coefficients the value of $C_F$ is either known or well-estimated. Stewart and Xiao \cite{stewart-xiao} computed $C_F$ when $F$ is a binomial form of degree at least three, while Xiao \cite[Theorems 3.1 and 4.1]{xiao} computed $C_F$ in the case when $F$ is a cubic form or a quartic form of non-zero discriminant. Fouvry and Waldschmidt \cite{fouvry-waldschmidt} estimated $C_{\Phi_n}$ for cyclotomic binary forms $\Phi_n(x, y)$. In \cite{mosunov2}, the author estimated $C_F$ for Chebyshev binary forms of the first kind $T_n(x, y)$ and of the second kind $U_n(x, y)$, as well as for the minimal binary forms $\Psi_n(x, y)$ of $2\cos \frac{2\pi}{n}$ and $\Pi_n(x, y)$ of $2\sin \frac{2\pi}{n}$.

In this article, we compute $C_F$ for two families of binary forms. For a positive integer $n$, define the binary forms $R_n(x, y)$ and $J_n(x, y)$ by means of the relation
\begin{equation} \label{eq:RnIn-relation}
(x + yi)^n = R_n(x, y) + J_n(x, y)i,
\end{equation}
where $x$ and $y$ are real variables. Examples of such forms are given in \mbox{Table \ref{tab:RnIn}}. Let $B(x, y)$ denote the beta function and for a positive integer $m$ let $\nu_2(m)$ denote its $2$-adic order. We prove the following.

\begin{table}[b]
\centering
\begin{tabular}{c | c | c}
$n$ & $R_n(x, y)$ & $J_n(x, y)$\\
\hline
$1$ & $x$ & $y$\\
$2$ & $x^2 - y^2$ & $2xy$\\
$3$ & $x^3 - 3xy^2$ & $3x^2y - y^3$\\
$4$ & $x^4 - 6x^2y^2 + y^4$ & $4x^3y - 4xy^3$\\
$5$ & $x^5 - 10x^3y^2 + 5xy^4$ & $5x^4y - 10x^2y^3 + y^5$\\
$6$ & $x^6 - 15x^4y^2 + 15x^2y^4 - y^6$ & $6x^5y - 20x^3y^3 + 6xy^5$\\
$7$ & $x^7 - 21x^5y^2 + 35x^3y^4 - 7xy^6$ & $7x^6y - 35x^4y^3 + 21x^2y^5 - y^7$\\
$8$ & $x^8 - 28x^6y^2 + 70x^4y^4 - 28x^2y^6 + y^8$ & $8x^7y - 56x^5y^3 + 56x^3y^5 - 8xy^7$
\end{tabular}
\caption{Binary forms $R_n(x, y)$ and $J_n(x, y)$ for $n = 1, 2, \ldots, 8$.}
\label{tab:RnIn}
\end{table}

\begin{thm} \label{thm:RnIn}
Let $n \geq 3$ be an integer. Then
$$
A_{R_n} = B\left(\frac{1}{2} - \frac{1}{n}, \frac{1}{2}\right), \quad W_{R_n} = \begin{cases}
1/2 & \text{if $n$ is odd,}\\
1/4 & \text{if $n$ is even and $4 \nmid n$,}\\
1/8 & \text{if $4 \mid n$,}
\end{cases}
$$
$$
C_{R_n} = W_{R_n}A_{R_n} = 2^{-\min(\nu_2(2n), 3)}B\left(\frac{1}{2} - \frac{1}{n}, \frac{1}{2}\right),
$$
and
$$
A_{J_n} = B\left(\frac{1}{2} - \frac{1}{n}, \frac{1}{2}\right), \quad W_{J_n} =
\begin{cases}
1/2 & \text{if $n$ is odd,}\\
1/4 & \text{if $n$ is even,}
\end{cases}
$$
$$
C_{J_n} = W_{J_n}A_{J_n} = 2^{-\min(\nu_2(2n), 2)}B\left(\frac{1}{2} - \frac{1}{n}, \frac{1}{2}\right).
$$
\end{thm}

The article is structured as follows. In Section \ref{sec:properties} we summarize some important properties of $R_n(x, y)$ and $J_n(x, y)$. In Sections \ref{sec:In-area} and \ref{sec:In-aut} we compute $A_{J_n}$ and $W_{J_n}$, respectively. In Sections \ref{sec:Rn-area} and \ref{sec:Rn-aut} we compute $A_{R_n}$ and $W_{R_n}$, respectively.

\section{Properties of $R_n(x, y)$ and $J_n(x, y)$} \label{sec:properties}

Let $n$ be a positive integer, and recall that the binary forms $R_n(x, y)$ and $J_n(x, y)$ are defined by means of the relation (\ref{eq:RnIn-relation}). In this section we will investigate some properties of $R_n(x, y)$ and $J_n(x, y)$. We begin with the observation that the coefficients of $R_n(x, y)$ and $J_n(x, y)$ are known.

\begin{prop} \label{prop:RnIn-coefficients}
Let $n$ be a positive integer. Then
$$
\begin{array}{r l}
R_n(x, y) & = \sum\limits_{\substack{0 \leq k \leq n\\ \text{$k$ even}}}(-1)^{\frac{k}{2}}\binom{n}{k}x^{n - k}y^k,\\
J_n(x, y) & = \sum\limits_{\substack{1 \leq k \leq n\\ \text{$k$ odd}}}(-1)^{\frac{k - 1}{2}}\binom{n}{k}x^{n - k}y^k.
\end{array}
$$
\end{prop}

\begin{proof}
Apply the Binomial Theorem to $(x + yi)^n$ in (\ref{eq:RnIn-relation}).
\end{proof}

Next, we observe that $R_n(x, y)$ and $J_n(x, y)$ split into linear factors over $\mathbb R$.

\begin{prop} \label{prop:RnIn-roots}
Let $n$ be a positive integer. Then
$$
\begin{array}{r l}
R_n(x, y) & = 2^{n - 1}\prod\limits_{k = 0}^{n - 1}\left(x\sin\left(\frac{(2k + 1)\pi}{2n}\right) - y\cos\left(\frac{(2k + 1)\pi}{2n}\right)\right),\\
J_n(x, y) & = 2^{n - 1}\prod\limits_{k = 1}^n\left(x\sin\frac{k\pi}{n} - y\cos\frac{k\pi}{n}\right).
\end{array}
$$
\end{prop}

\begin{proof}
Setting $x = \cos \theta$ and $y = \sin \theta$ in (\ref{eq:RnIn-relation}), it follows from de Moivre's formula that
$$
\cos(n\theta) + \sin(n\theta) i = (\cos \theta + \sin \theta i)^n = R_n(\cos \theta, \sin \theta) + J_n(\cos \theta, \sin \theta) i.
$$
Thus,
$$
R_n(\cos \theta, \sin \theta) = \cos(n \theta) \quad \text{and} \quad J_n(\cos \theta, \sin \theta) = \sin(n \theta)
$$
for all $\theta \in \mathbb R$.

Notice that $\cos(n\theta) = 0$ for $\theta = \frac{(2k + 1)\pi}{2n}$, where $k = 0, 1, \ldots, n - 1$. Since the zeroes of $\cos(n\theta)$ are known, it is now possible to write down $R_n(x, y)$ as a product of $n$ linear factors over $\mathbb R$ times a non-zero constant $r_n$. Now, observe that
\small
\begin{align*}
\prod\limits_{k = 0}^{n - 1}\sin\left(\frac{(2k + 1)\pi}{2n}\right)
& = \prod\limits_{k = 0}^{n - 1}\frac{e^{\frac{(2k + 1)\pi i}{2n}} - e^{-\frac{(2k + 1)\pi i}{2n}}}{(2i)} = (-2i)^{-n}\prod\limits_{k = 0}^{n - 1}e^{-\frac{(2k + 1)\pi i}{2n}}\prod\limits_{k = 0}^{n - 1}\left(1 - e^{\frac{(2k + 1)\pi i}{n}}\right)\\
& = (-2i)^{-n}e^{-\frac{n\pi i}{2}}\prod\limits_{k = 0}^{n - 1}\left(1 - e^{\frac{(2k + 1)\pi i}{n}}\right) = 2^{-n}\prod\limits_{k = 0}^{n - 1}\left(1 - e^{\frac{(2k + 1)\pi i}{n}}\right)= 2^{1 - n},
\end{align*}
\normalsize
where the last equality follows from the fact that $\prod_{k = 0}^{n - 1}\left(x - e^{\frac{(2k + 1)\pi i}{n}}\right) = x^n + 1$. From Proposition \ref{prop:RnIn-coefficients} we know that the coefficient of $x^n$ in $R_n(x, y)$ is always equal to $1$, so
$$
1 = r_n\prod\limits_{k = 0}^{n - 1}\sin\left(\frac{(2k + 1)\pi}{2n}\right) = r_n2^{1 - n}.
$$
Thus, $r_n = 2^{n - 1}$.

Notice that $\sin(n\theta) = 0$ for $\theta = \frac{k\pi}{n}$, where $k = 1, 2, \ldots, n$. Since the zeroes of $\sin(n\theta)$ are known, it is now possible to write down $J_n(x, y)$ as a product of $n$ linear factors over $\mathbb R$ times a non-zero constant $s_n$. By \cite[Section 4]{mosunov1}, \mbox{$\prod_{k = 1}^{n - 1}\sin\frac{k\pi}{n} = 2^{1 - n}n$}. From Proposition \ref{prop:RnIn-coefficients} we know that the coefficient of $x^{n - 1}y$ in $J_n(x, y)$ is always equal to $n$, so
$$
n = s_n\prod\limits_{k = 1}^{n - 1}\sin\frac{k\pi}{n} = s_n2^{1 - n}n.
$$
Thus, $s_n = 2^{n - 1}$.
\end{proof}

\section{Area of the Fundamental Region of $J_n(x, y)$} \label{sec:In-area}

For a positive integer $n$, define
$$
F_n^*(x, y) = \prod\limits_{k = 1}^n\left(x\sin\frac{k\pi}{n} - y\cos\frac{k\pi}{n}\right).
$$
Then it follows from Proposition \ref{prop:RnIn-roots} that $J_n(x, y) = 2^{n - 1}F_n^*(x, y)$. By \mbox{\cite[Theorem 1]{bean-laugesen}},
$$
A_{F_n^*} = 4^{1 - 1/n}B\left(\frac{1}{2} - \frac{1}{n}, \frac{1}{2}\right)
$$
for every integer $n \geq 3$. Further, for any binary form $F$ of positive degree $d$, the formula $A_F = \int_{-\infty}^{+\infty}|F(x, 1)|^{-2/d}dx$ implies that $A_{cF} = |c|^{-2/d}A_F$ for any non-zero complex number $c$ (see \cite{bean94}). Since $\deg J_n = \deg F_n^* = n$,
$$
A_{J_n} = A_{2^{n - 1}F_n^*} = 2^{-2(n - 1)/n}A_{F_n^*} = B\left(\frac{1}{2} - \frac{1}{n}, \frac{1}{2}\right)
$$
for any integer $n \geq 3$.

\section{Rational Automorphism Group of $J_n(x, y)$} \label{sec:In-aut}

For a binary form $F$ with integer coefficients, degree $d \geq 3$ and non-zero discriminant, define
$$
\operatorname{Aut} |F| = \left\{A \in M_{2 \times 2}(\mathbb Q)\ \colon\ F_A = \pm F\right\}.
$$
Then $\operatorname{Aut} |F|$ is a finite subgroup of $\operatorname{GL}_2(\mathbb Q)$, so it is $\operatorname{GL}_2(\mathbb Q)$-conjugate to one of the ten groups in \mbox{Table \ref{tab:finite-subgroups}}. Furthermore, $\operatorname{Aut} F$ is a normal subgroup of $\operatorname{Aut} |F|$. Let
\begin{equation} \label{eq:D2}
D_2 = \left\langle \begin{pmatrix}-1 & 0\\0 & 1\end{pmatrix}, \begin{pmatrix}1 & 0\\0 & -1\end{pmatrix}\right\rangle,
\end{equation}
and let $\mathbf D_2$ and $\mathbf D_4$ be as in \mbox{Table \ref{tab:finite-subgroups}}. Note that the group $D_2$, while being equivalent to $\mathbf D_2$ under $\operatorname{GL}_2(\mathbb Q)$-conjugation, is not equal to $\mathbf D_2$. In this section we will prove the following lemma.

\begin{lem} \label{lem:AutIn}
Let $n \geq 3$ be an integer.
\begin{enumerate}[1.]
\item If $n$ is odd, then
$$
\operatorname{Aut} J_n = \left\langle \begin{pmatrix}-1 & 0\\0 & 1\end{pmatrix}\right\rangle, \quad \operatorname{Aut} |J_n| = D_2, \quad W_{J_n} = \frac{1}{2}.
$$

\item If $n$ is even, then
$$
\operatorname{Aut} J_n = \mathbf D_2, \quad \operatorname{Aut} |J_n| = \mathbf D_4, \quad W_{J_n} = \frac{1}{4}.
$$
\end{enumerate}
\end{lem}

\begin{proof}
Suppose that $n$ is even. Then it follows from Proposition \ref{prop:RnIn-coefficients} that $J_n(x, y) = xyG_n(x^2, y^2)$ for some binary form $G_n(x, y)$, with integer coefficients, which is either \emph{reciprocal} ($4 \nmid n$) or \emph{skew-reciprocal} ($4 \mid n$). Thus, $\mathbf D_4 \subseteq \operatorname{Aut} |J_n|$. Since we know what finite subgroups of $\operatorname{GL}_2(\mathbb Q)$ are (see \mbox{Table \ref{tab:finite-subgroups}}), we can immediately conclude that $\operatorname{Aut} |J_n| = \mathbf D_4$. By checking every element of $\mathbf D_4$, we can also conclude that $\operatorname{Aut} J_n = \mathbf D_2$. Since $\operatorname{Aut} J_n \subseteq \operatorname{GL}_2(\mathbb Z)$, it follows from (\ref{eq:WF}) that \mbox{$W_{J_n} = |\operatorname{Aut} J_n|^{-1} = \frac{1}{4}$}.

Suppose that $n$ is odd. Then it follows from Proposition \ref{prop:RnIn-coefficients} that $J_n(x, y) = yG_n(x^2, y^2)$ for some binary form $G_n(x, y)$ with integer coefficients. Thus, $D_2 \subseteq \operatorname{Aut} |J_n|$. We claim that, in fact, the equality holds. Assume for a contradiction that \mbox{$D_2 \subsetneq \operatorname{Aut} |J_n|$}. By \cite[Lemma 3.4]{mosunov2}, there must exist a non-zero rational number $t$ such that either
$$
\operatorname{Aut} |J_n| = \left\langle\begin{pmatrix}-1 & 0\\0 & 1\end{pmatrix}, \begin{pmatrix}0 & t\\-1/t & 0\end{pmatrix}\right\rangle
$$
or
$$
\operatorname{Aut} |J_n| = \left\langle\begin{pmatrix}-1 & 0\\0 & 1\end{pmatrix}, \begin{pmatrix}1/2 & t/2\\-3/(2t) & 1/2\end{pmatrix}\right\rangle.
$$
Let us eliminate each of these two options.

\begin{enumerate}[1.]
\item Suppose that there exists a non-zero rational number $t$ such that \mbox{$P \in \operatorname{Aut} |J_n|$}, where $P =  \begin{pmatrix}0 & t\\-1/t & 0\end{pmatrix}$. Then
$$
J_n(x, y) = \pm J_n\left(ty, -t^{-1}x\right).
$$
By plugging $x = 1$ and $y = 0$ into the above equation, we see that
$$
0 = (-1)^{\frac{n - 1}{2}}(-t^{-1})^n,
$$
which is impossible, because $t \neq 0$.

\item Suppose that there exists a non-zero rational number $t$ such that \mbox{$P \in \operatorname{Aut} |J_n|$}, where $P = \begin{pmatrix}1/2 & t/2\\-3/(2t) & 1/2\end{pmatrix}$. Then
$$
J_n(x, y) = \pm J_n\left(\frac{1}{2}x + \frac{t}{2}y,\ -\frac{3}{2t}x + \frac{1}{2}y\right).
$$
By plugging $x = 1$ and $y = 0$, we find that
$$
J_n\left(\frac{1}{2},\ -\frac{3}{2t}\right) = 0.
$$
But then it follows from Proposition \ref{prop:RnIn-coefficients} that
$$
\frac{1}{2}\sin\frac{k\pi}{n} + \frac{3}{2t}\cos\frac{k\pi}{n} = 0
$$
for some integer $k \in \{1, \ldots, n\}$. Thus, $-\frac{t}{3} = \cot\frac{k\pi}{n}$, where $k \neq n$. But the only possible rational values of $\cot\frac{k\pi}{n}$ are $0$ and $\pm 1$, and since $t \neq 0$, we find that $\cot\frac{k\pi}{n} = \pm 1$. Since $\cot(x) = \pm 1$ if and only if $x = \frac{\pi}{4} + \frac{\pi}{2}m$ for some $m \in \mathbb Z$, we conclude that $\frac{k\pi}{n} = \frac{\pi}{4} + \frac{\pi}{2}m$. But then $4k = (1 + 2m)n$, and so we reach a contradiction, because the number on the left-hand side is even, while the number on the right-hand side is odd.
\end{enumerate}

In view of the above results we can conclude that $\operatorname{Aut} |J_n| = D_2$. By checking every element of $D_2$, we deduce that $\operatorname{Aut} J_n = \left\langle \begin{pmatrix}-1 & 0\\0 & 1\end{pmatrix}\right\rangle$. Since $\operatorname{Aut} J_n \subseteq \operatorname{GL}_2(\mathbb Z)$, it follows from (\ref{eq:WF}) that \mbox{$W_{J_n} = |\operatorname{Aut} J_n|^{-1} = \frac{1}{2}$}.
\end{proof}

\section{Area of the Fundamental Region of $R_n(x, y)$} \label{sec:Rn-area}

For a positive integer $n$, let
$$
P_n = \begin{pmatrix}\cos\frac{\pi}{2n} & \sin\frac{\pi}{2n}\\ -\sin \frac{\pi}{2n} & \cos\frac{\pi}{2n}\end{pmatrix}
$$
be the matrix corresponding to the clockwise rotation in $\mathbb R^2$ by $\frac{\pi}{2n}$ radians about the origin. Notice that the determinant of $P_n$ is equal to $1$. By \mbox{Proposition \ref{prop:RnIn-roots}},
\small
\begin{align*}
(J_n)_{P_n}(x, y)
& = J_n\left(x\cos\frac{\pi}{2n} + y\sin\frac{\pi}{2n},\ -x\sin\frac{\pi}{2n} + y\cos\frac{\pi}{2n}\right)\\
& = -2^{n - 1}\prod\limits_{k = 0}^{n - 1}\left(\sin\frac{k\pi}{n}\left(x\cos\frac{\pi}{2n} + y\sin\frac{\pi}{2n}\right) - \cos\frac{k\pi}{n}\left(-x\sin\frac{\pi}{2n} + y\cos\frac{\pi}{2n}\right)\right)\\
& = -2^{n - 1}\prod\limits_{k = 0}^{n - 1}\left(x\sin\left(\frac{(2k + 1)\pi}{2n}\right) - y\cos\left(\frac{(2k + 1)\pi}{2n}\right)\right)\\
& = -R_n(x, y).
\end{align*}
\normalsize
In \cite{bean94}, Bean proved that $A_F = |\det P| A_{F_P}$ for any $P \in \operatorname{GL}_2(\mathbb R)$. Using this formula, we find that the area of the fundamental region of $R_n(x, y)$ is equal to
$$
A_{R_n} = A_{-R_n} = A_{(J_n)_{P_n}} = |\det P_n|^{-1}A_{J_n} = A_{J_n} = B\left(\frac{1}{2} - \frac{1}{n},\ \frac{1}{2}\right).
$$

\section{Rational Automorphism Group of $R_n(x, y)$}  \label{sec:Rn-aut}

Recall the definition of $D_2$ from (\ref{eq:D2}), as well as the definitions of $\mathbf D_2$ and $\mathbf D_4$ from \mbox{Table \ref{tab:finite-subgroups}}. In this section we will prove the following lemma.

\begin{lem} \label{lem:AutIn}
Let $n \geq 3$ be an integer.
\begin{enumerate}[1.]
\item If $n$ is odd, then
$$
\operatorname{Aut} R_n = \left\langle \begin{pmatrix}1 & 0\\0 & -1\end{pmatrix}\right\rangle, \quad \operatorname{Aut} |R_n| = D_2, \quad W_{R_n} = \frac{1}{2}.
$$

\item If $n$ is even and $4 \nmid n$, then
$$
\operatorname{Aut} R_n = \mathbf D_2, \quad \operatorname{Aut} |R_n| = \mathbf D_4, \quad W_{R_n} = \frac{1}{4}.
$$

\item If $4 \mid n$, then
$$
\operatorname{Aut} R_n = \operatorname{Aut} |R_n| = \mathbf D_4, \quad W_{R_n} = \frac{1}{8}.
$$
\end{enumerate}
\end{lem}

\begin{proof}
Suppose that $n$ is even. Then it follows from Proposition \ref{prop:RnIn-coefficients} that $R_n(x, y) = G_n(x^2, y^2)$ for some binary form $G_n(x, y)$, with integer coefficients, which is either \emph{reciprocal} ($4 \mid n$) or \emph{skew-reciprocal} ($4 \nmid n$). Thus, $\mathbf D_4 \subseteq \operatorname{Aut}|R_n|$, and since we know what finite subgroups of $\operatorname{GL}_2(\mathbb Q)$ are (see \mbox{Table \ref{tab:finite-subgroups}}), we can immediately conclude that $\operatorname{Aut} |R_n| = \mathbf D_4$. By checking every element of $\mathbf D_4$, we can also conclude that $\operatorname{Aut} R_n = \mathbf D_4$ when $4 \mid n$ and $\operatorname{Aut} R_n = \mathbf D_2$ when $4 \nmid n$. Since $\operatorname{Aut} R_n \subseteq \operatorname{GL}_2(\mathbb Z)$, it follows from (\ref{eq:WF}) that \mbox{$W_{R_n} = \frac{1}{8}$} when $4 \mid n$ and \mbox{$W_{R_n} = \frac{1}{4}$} when $4 \nmid n$.

Suppose that $n$ is odd. Then it follows from Proposition \ref{prop:RnIn-coefficients} that $R_n(x, y) = xG_n(x^2, y^2)$ for some binary form $G_n(x, y)$ with integer coefficients. We claim that, in fact, the equality holds. Assume for a contradiction that \mbox{$D_2 \subsetneq \operatorname{Aut} |R_n|$}. By \cite[Lemma 3.4]{mosunov2}, there must exist a non-zero rational number $t$ such that either
$$
\operatorname{Aut} |R_n| = \left\langle\begin{pmatrix}-1 & 0\\0 & 1\end{pmatrix}, \begin{pmatrix}0 & t\\-1/t & 0\end{pmatrix}\right\rangle
$$
or
$$
\operatorname{Aut} |R_n| = \left\langle\begin{pmatrix}-1 & 0\\0 & 1\end{pmatrix}, \begin{pmatrix}1/2 & t/2\\-3/(2t) & 1/2\end{pmatrix}\right\rangle.
$$
Let us eliminate each of these two options.

\begin{enumerate}[1.]
\item Suppose that there exists a non-zero rational number $t$ such that \mbox{$M \in \operatorname{Aut} |R_n|$}, where $P =  \begin{pmatrix}0 & t\\-1/t & 0\end{pmatrix}$. Then
$$
R_n(x, y) = \pm R_n\left(ty, -t^{-1}x\right).
$$
By plugging $x = 0$ and $y = 1$ into the above equation, we see that $t^n = 0$, which is impossible, because $t \neq 0$.

\item Suppose that there exists a non-zero rational number $t$ such that \mbox{$M \in \operatorname{Aut} |R_n|$}, where $P = \begin{pmatrix}1/2 & t/2\\-3/(2t) & 1/2\end{pmatrix}$. Then
$$
R_n(x, y) = \pm R_n\left(\frac{1}{2}x + \frac{t}{2}y,\ -\frac{3}{2t}x + \frac{1}{2}y\right).
$$
By plugging $x = 0$ and $y = 1$, we find that
$$
R_n\left(\frac{t}{2},\ \frac{1}{2}\right) = 0.
$$
But then it follows from Proposition \ref{prop:RnIn-coefficients} that
$$
\frac{t}{2}\sin\left(\frac{(2k + 1)\pi}{2n}\right) - \frac{1}{2}\cos\left(\frac{(2k + 1)\pi}{2n}\right) = 0
$$
for some integer $k \in \{0, \ldots, n - 1\}$. Thus, $t = \cot\left(\frac{(2k + 1)\pi}{2n}\right)$. But the only possible rational values of $\cot\left(\frac{(2k + 1)\pi}{2n}\right)$ are $0$ and $\pm 1$, and since $t \neq 0$, we find that $\cot\left(\frac{(2k + 1)\pi}{2n}\right) = \pm 1$. Since $\cot(x) = \pm 1$ if and only if $x = \frac{\pi}{4} + \frac{\pi}{2}m$ for some $m \in \mathbb Z$, we conclude that $\frac{(2k + 1)\pi}{2n} = \frac{\pi}{4} + \frac{\pi}{2}m$. But then $4k + 2 = (1 + 2m)n$, and so we reach a contradiction, because the number on the left-hand side is even, while the number on the right-hand side is odd.
\end{enumerate}

In view of the above results we can conclude that $\operatorname{Aut} |R_n| = D_2$. By checking every element of $D_2$, we deduce that $\operatorname{Aut} R_n = \left\langle \begin{pmatrix}1 & 0\\0 & -1\end{pmatrix}\right\rangle$. Since $\operatorname{Aut} R_n \subseteq \operatorname{GL}_2(\mathbb Z)$, it follows from (\ref{eq:WF}) that \mbox{$W_{R_n} = |\operatorname{Aut} R_n|^{-1} = \frac{1}{2}$}.
\end{proof}

\section*{Acknowledgements}

The author is grateful to the anonymous reviewer for their advice on how to improve the article.

\end{document}